\documentclass[11pt,leqno]{article}
\usepackage{amsmath,amssymb,amsfonts,amsthm,bbm,mathrsfs,verbatim} 
\usepackage[misc]{ifsym}
\usepackage{graphics}                 
\usepackage{graphicx}
\usepackage{float}
\usepackage{array}
\usepackage{color}                    
\usepackage{hyperref}                 
\usepackage{makeidx}
\usepackage[numbers]{natbib}

\newtheorem{theorem}{Theorem}[section]
\newtheorem{proposition}[theorem]{Proposition}
\newtheorem{corollary}[theorem]{Corollary}
\newtheorem{lemma}[theorem]{Lemma}
\theoremstyle{definition}
\newtheorem{remark}[theorem]{Remark}
\newtheorem{definition}[theorem]{Definition}

\newtheorem{assumption}[theorem]{Assumption}

\def\R{{\mathbb R}}
\def\N{{\mathbb N}}

\def\Rp{\R_{+}}

\def\AK{{\cal A}(K)}
\def\AKn{{\cal A}(K_n)}
\def\Lil{L^{\infty}_{loc}(\Rp)}

\def\Linl{L^{\infty}(\Rp)}

\def\Ml{{\cal M}_{loc}(\Rp)}
\def\MII{{\cal M}(I)}
\def\BVl{BV_{loc}(\Rp)}
\def\kdp{(k')^{+}}
\def\kdm{(k')^{-}}
\def\ws{\stackrel{\ast}{\rightharpoonup}}

\def\S{{\mathcal S}}
\def\T{{\mathcal T}}
\def\M{{\mathcal O}}
\def\fe{f^{(h)}}
\def\Ue{U^{(h)}}
\def\Ued{U^{(h)\prime}}
\def\ke{k^{(h)}}
\def\ked{k^{(h)\prime}}

\def\ce{c^{(h)}}

\newcommand{\deb}{\rightharpoonup}

\newcommand{\ve}{\varepsilon}

\newcommand*{\avint}{\mathop{\ooalign{$\int$\cr$-$}}}

\begin{document}

\title{Spatial rational expectations equilibria\\ in the Ramsey model of optimal growth}

\author{F. Santambrogio\thanks{Laboratoire de Math\'ematiques d'Orsay, 
Univ. Paris-Sud, CNRS, Universit\'e Paris-Saclay, 
91405 Orsay Cedex, FRANCE, 
{\tt  filippo.santambrogio@math.u-psud.fr}}, A. Xepapadeas\thanks{Department of International and European Economic Studies, Athens University of Economics and Business, 76 Patission Str., 10434 Athens, GREECE, and Department of Economics, University of Bologna, P.zza Scaravilli 2, 40126 Bologna, ITALY, {\tt xepapad@aueb.gr}} ,  A. N. Yannacopoulos\thanks{Department of Statistics, Athens University of Economics and Business, 95 Patission Str., 10434 Athens, GREECE, {\tt ayannaco@aueb.gr }}}

\date{}
\maketitle

\begin{abstract}
It is the aim of this work  provide a rigorous treatment concerning the formation of spatial rational expectations equlibria in a general class of spatial economic models under the effect of externalities, using techniques from the calculus of variations.  Using detailed estimates for a parametric optimisation problem, the existence of spatial rational expectations equilibria  is proved and they are characterised in terms of a nonlocal Euler-Lagrange equation.
\end{abstract}

\section{Introduction}

We consider in this work infinite horizon spatial optimal growth models, where
capital accumulation ocurres localy but output production is affected by
general spatial interactions. In these models a natural research question is the existence of a
rational expectations equlibrium (REE) where a reperesentative individual in
each location acting as a `` local planner''
 maximises discounted utility of consumption by considering the spatial
externality as exogenous.

In this paper we assume that spatial interactions among locations can be
expressed as a spatial externality which in general attenuates with distance.
In this context spatial interactions can be regarded, for example, as
knowledge spillover effects from one location to another. The main idea
associated with knowledge spillovers is that innovation and new productive
knowledge flows more easily among agents which are located within the same
area (e.g. \cite{krugman1991geography}, \cite{breschi2001localised}). Thus
proximity is important in characterizing spatial spillovers
(\cite{baldwin2004agglomeration}, \cite{breinlich2013regional}). We incorporate
general spatial spillovers by interpreting the capital stock of each firm in a
broad sense to include knowledge along with physical capital (e.g.
\cite{romer1986increasing}). As argued by \cite{quah2002spatial} the effect of
capital on each firm's output, at any given point in time, does not depend
just on the accumulated stock by the firm up to this time, but on capital
accumulated in nearby locations by other firms. Thus the spatial externality
takes the form of a Romer (\cite{romer1986increasing}) type externality where,
by keeping all other factors in fixed supply, output is determined by own
capital stock and by an appropriately defined aggregate of capital stocks of
firms across the spatial domain. The capital stock aggregate is determined by
a distance-response function (see \cite{papageorgiou1983agglomeration}
for an early use of distance - response functions) that measures the strength
of the effect on the output of a firm in a certain location induced by the
capital stock accumulated by a firm in another location.

A positive distance-response function that attenuates with distance can be
interpreted as reflecting knowledge spillovers. A distance-response which is
negative indicates a negative externality such as generalised congestion
effects. Thus, by combining a distance-response function, centripetal and
centrifugal responses can be introduced with the  strength - positive or
negative - of these forces diminishing with distance.

Modeling the spatial spillovers by a general nonlinear integral operator, we
analyze the problem where forward-looking local planners maximise discounted
utility by choosing a local consumption path, subject to local capital
accumulation, by considering the spatial externality affecting there local
production as parametric. A REE in the whole spatial domain is defined as the
equilibrium local capital stock and consuption paths which emerge when the
capital stocks comprising the externality are determined endogenously through
the optimality conditions in each location.

The aim of this paper is to provide a rigorous mathematical study of rational expectation 
equilibria  emerging in a
 general growth model with spatial interactions, providing existence results  using techniques from the calculus of
variations in BV spaces and appropriate fixed point theorems. To the best of our knowledge such an analysis has not appeared in the literature before. Our detailed analysis highlights
the effects of conditions on the primitives of the economy on the behaviour of such equilibria and  clarifies aspects related to the
fine qualitative properties of the optimal path and optimal policy. Furthermore, the characterization of the rational expectations equilibria  is obtained in terms of a nonlocal Euler-Lagrange equation that may be used for further study of the qualitative spatio-temporal properties of the optimal capital alocations.
We think that the results of our paper shed some light towards a better understanding of the problem of
spatial growth, which is receiving increasing attention in recent economic literature (see e.g. \cite{brock2014optimal} or \cite{boucekkine2013spatial}). 

On the other hand the problem of optimal growth in its temporal version has attracted the attention of the more mathematically oriented literature in various forms. For example Ekeland has studied in detail and rigorously various aspects of the optimal growth problem, focusing on well-posedness as well as on the derivation of transversality conditions for the infinite horizon version of the model (see e.g. \cite{ekeland2010frank} and references therein).  Techniques from the theory of calculus of variations in Sobolev spaces have been employed for the study of the temporal problem (see e.g.  \cite{kravvaritis1991existence} or \cite{ le2007optimal}; see also \cite{sagara2007nonconvex} for the case of recursive utilities) whereas techniques from the theory of Hamilton-Jacobi equation and viscosity solutions have been used for the study of stochastic effects in the Ramsey model (see e.g. \cite{adachi2009optimal} and \cite{morimoto2008optimal}). However, to the best of our knowledge a rigorous mathematical treatment, with the aim of studying equilibrium questions, of the problem of optimal growth in spatial economies has not been addressed so far. 

It is the aim of this paper to make a few steps in this direction by presenting a rigorous treatment of a spatial economic growth model in the presence of spatial externalities, following up on previous work by Brock et al \cite{brock2014optimal}. In particular, we present a proof of existence of spatial rational expectations equilbria (SREE)  under rather general assumptions and furthermore we characterise them in terms of a nonlocal Euler-Lagrange equation. The treatment of the existence of SREE requires some detailed (and rather technical) results for a parametric individual optimisation problem in which the state of the externalities is assumed as a 
known functional parameter, which can be interesting in their own right. Indeed, these parametric problems are exactly a non-autonomous optimal growth (Ramsey) problem. To analyse it, it is prefarable to bypass the usual way of treating the autonomous problem, which consists in building (or guessing) solutions to suitable sufficient optimality conditions (which are expressed in terms of differential equations and transversality conditions), as it is not always easy to solve them. Thus, one has to follow the direct method of calculus of variations, i.e. first proving existence in a very weak and wide functional space (the space of BV functions, in our case) and then prove regularity and properties about the optimisers. For example, proving lower bounds on the optimal path, which is nowadays standard in the autonomous case, requires more attention in the setting of our paper, and is one of the key technical parts  of our work.

Once some desirable properties of the solution of the  parametric individual problem have been established, the existence of  SREE is obtained by an application of the Schauder fixed point theorem in an appropriate functional space setting. The establishment of the necessary estimates for the individual problem as well as the derivation of the nonlocal Euler-Lagrange characterization of the SREE is based on a Pontryagin maximum principle. Importantly, the Euler-Lagrange equation which characterises the SREE can be used for either for the numerical calculation of the SREE or for the derivation of qualitative aspects of the SREE such as for instance pattern formation behaviour.

The paper is organised as follows. In Section 2, a general model for optimal economic growth under spatial externalities is proposed. In Section 3 a detailed study of the  individual optimisation problem, where the externalities are taken as a given parameter, is performed. Using the results of Section 3, in Section 4 we study the problem of existence of REE and their characterization in terms of the Euler-Lagrange equation and provide the desired existence result.

\section{The model}

We consider a spatial growth model with spatial externalities. In particular let $D$ be a given geographic region (one can think at it as a compact subset of $\mathbb R^2$, but this is not crucial, and a more general metric space could also be used) and let the function $k : \Rp \times D \to \Rp$  model the spatio-temporal distribution of capital stock in the region. We are in a one all purpose commodity world we produce output which can be used for consumption or investment, that is accumulation of capita.  The production of output depends on spatial externalities, which display spatio-temporal variability as well, and are modeled in terms of the function $K : \Rp \times D \to (0, \infty)$. Both negative or positive externalities are included in the model in the sense that externalities may increase or decrease local production. Capital is assumed to be immobile but spatial effects are introduced through the varying level of externality effects $K$.

The capital accumulation satisfies the differential equation 
\begin{eqnarray*}
k'(t,z) = f(k(t,z),K(t,z))-c(t,z)
\end{eqnarray*}
where $k'(t,z)=\frac{\partial k}{\partial t}(t,z)$, $f$ is the net output given by the capital and the externalities, and $c : \Rp \times D \to \Rp$ is a function describing spatio-temporal consumption. By net output we mean the production, net of depreciation: this means that we include capital depreciation in $f$, i.e., $f(k,K)=\varphi(k, K) - \delta k$. Here $\varphi$ is a standard neoclassical production function and the parameter $\delta>0$ gives capital loss due to depreciation. A possible example for the production function can be a Cobb-Douglas type production function $\varphi(k,K)=A k^{\alpha} K^{\beta}$ for appropriate values of $\alpha$ and $\beta$ (with $\alpha + \beta <1$), so that we would have $f(k,K)=A k^{\alpha} K^{\beta}-\delta k$ . To simplify the exposition we assume that in each location labour is fixed and fully employed. So the arguments in the production and net output functions can be regarded as per capita quantities.

The  externalities $K$ are determined (in a self-consistent fashion i.e. endogenously) by the capital allocation $k$. The endogeneity is modelled by assuming that the externality effects at positions  $z \in D$ and time $t \in \Rp$ are expressed as
 $K(t,z)=(\S k)(t,z)$ where  $\S$ is an operator modelling the positive or negative externality effects.  As a particular example, we can consider operators $\S$ that are based on integral operators  e.g.  an operator the form 
\begin{eqnarray}\label{defi SS}
(\S k)(t,z)=\psi\left(\int_{D} w(z,y)k(t,y)dy\right)
\end{eqnarray}
where $w : D \times D \to \R$ is a suitable kernel function and $\psi\R\to I\subset\R_+$ is a nonlinearity, or even 

\begin{eqnarray*}
(\S k)(t,z)=\psi\left(\int_{D} w(z,y) \psi_0(k(t,y))dy\right),
\end{eqnarray*}
for a suitable function $\psi_0 : \Rp \to \R$. We need to include nonlinearities in the model for the externalities because we want a model where $K$ is positive (in many cases, bounded from below by a strictly positive constant), but we also want to consider possible negative externalities.

More precisely we will consider an operator $\mathcal S$ satisfying the following condition
\begin{assumption}\label{ass on S}
Given an interval $\bar I\subset \R_+$, the operator $\mathcal S:L^\infty(D)\to C(D;\bar I)$ is a (possibly nonlinear) Lipschitz map for the $L^\infty$ topology and it maps bounded sets of $L^\infty(D)$ into compact subsets of $C(D;\bar I)$, the space of continuous functions over $D$, taking values in the interval $\bar I$, endowed with the topology of uniform convergence.
\end{assumption}
The choice of the interval $\bar I$, which could be unbounded (but typically is bounded away from $0$), will be made so that the net output function $f$ and the externalities $K$ satisfy some compatibility conditions which will be precised later on.\smallskip

We will then use this operator for each fixed time, taking $K(t,\cdot)=\mathcal S(k(t,\cdot))$.\bigskip

Consider now a local  planner at location $z \in D$. The planner can only formulate expectations concerning the evolution of the externalities over the time horizon $(0,\infty)$, and his/her actions  directly influence  only the local variables taking the externality as parametric.  Given his/her expectation of the exogenous  temporal evolution of the effect of  the externalities $K(\cdot, z)$ at location $z$,  he/she maximises local intertemporal utility of consumption at location $z \in D$, i.e. solves the (parametric) maximization problem
\begin{equation}\label{1}
\begin{aligned}
& \max \int_{0}^{\infty}e^{-rt} U(c(t)) dt,  \,\,\, \mbox{subject to}\\
& k'(t,z) = f(k(t,z),K(t,z))-c(t,z) \\
& k(t,z) \ge 0, \,\,\, c(t,z) \ge 0,\,\,\, (t,z) \,\,\, a.e,
\end{aligned}
\end{equation}
where $U$ is an appropriate utility function and $r>0$ is a utility discount rate. The solution of problem \eqref{1} yields  an optimal consumption rule $c^{*}(\cdot ,z)=c^{*}(\cdot,z; K)$ and an optimal path $k^{*}(\cdot,z)=k^{*}(\cdot, z ;K)$, where we include $K$ in the notation to emphasise that this optimal path depends on the  exogenous externality state $K$ faced by the ``representative'' planner at location $z$. Each planner at location $z \in D$ solves a version of problem \eqref{1} (for the proper choice of $z$,  thus obtaining a function $k^{*}(\cdot,\cdot ; K) : \Rp \times D \times \Rp$, and a function $c^{*}(\cdot,\cdot ; K) : \Rp \times D \to \Rp$, such that $k^{*}(\cdot ,z ;K) : \Rp \to \Rp$ and $c^{*}(\cdot , z ; K) : \Rp \to \Rp$ are the optimal path and the optimal and the optimal consumption of the ``representative'' planner at location $z \in D$, given the instantaneous state of the externalities $K : \Rp \times D \to \R$.  Recall our assumption that externalities are determined endogenously in terms of the operator $\S$. This implies that if all individual  agents, at any $z \in D$ are well informed concerning the state of the externalities $K$, and using this knowledge solve their individual problems \eqref{1}  determining the optimal function $k^{*})\cdot, \cdot , K)$, then consistency of the model imposes that $K = \S k^{*}$.

The above discussion leads to the definition of the operator $\T$ by $K \mapsto k^{*}(\cdot, \cdot, K) \mapsto \S k^{*}(\cdot, \cdot, K)$, in terms of which  we may now define a rational expectations equilibrium for the spatial economy.

\begin{definition}
A fixed point of the operator $\T$, defined by  $K \mapsto k^{*}(\cdot, \cdot, K) \mapsto \S k^{*}(\cdot, \cdot, K)$  is a rational expectations equilibrium for the spatial economy.
\end{definition} 

\begin{remark}
Our definition of a rational expectations equilibrium for the spatial economy is motivated by the interpretation of the above scheme as a best-response scheme. In particular given that the agent at location $z \in D$, anticipates that the externality effects at $z \in D$  are going to develop in the future as $K(\cdot, z)$ he/she designs the optimal path for the economy at this location as the best response to this anticipated externality effects, which is the solution $k^{*}(\cdot,z ; K(\cdot, z))$ to problem \eqref{1}. This is true for any agent located at any other site $z' \in D$. Their best responses to their anticipated level of externalities will clearly contribute to the formulation of the actual state of the externalities which will become $\S k^{*}(\cdot , \cdot ; K(\cdot, \cdot))$.  This actual state can be interpreted as an adjustment of the agents anticipation of the level of the externalities, hence the level of externalities where this adjustment procedure reacertains the agents anticipation is called a spatial rational expectations equilibrium.
\end{remark}

To this point we deliberately refrain from setting a detailed functional space setting for the problem, and we prefer to introduce this later on (see Sections \ref{INDIV-OPT} and \ref{RAT-EXP}) after the necessary technical estimates that clarify our choice are presented. Here, we only impose assumptions on the primitives of the problem and in particular to the production function and the utility function.

\begin{assumption}[ Assumptions on $f$] \label{f-ASSUMP}  The net output function $f$ satisfies:
\begin{itemize}
\item[$(i)$] $f$ is continuous on $\Rp\times\Rp$ and $f(0,K)=0$ for every $K \in \bar I$ (where $\bar I$ is the interval in \ref{ass on S}).
\item[$(ii)$]  $f$ is $C^1$ on $(0,+\infty)\times (0,+\infty)$ and by $f_{k}$, $f_{K}$ we denote the two partial derivatives.
\item[$(iii)$]  there exists a constant $\delta>0$ such that $f_k\geq -\delta$, and moreover $f_k$ is bounded from above on any set of the form $[a,\infty)\times [0,c]$, for $a>0$ and $c>0$.
\item[$(iv)$]  for every $M$ there is $\bar k(M)$ such that $f(k,K)\le 0$ for every $K\le M$ and $k\ge \bar k(M)$; moreover, $\bar k(M)=o(M)$ as $M\to\infty$.
\item[$(v)$]  there exists $k_1>0$ such that $f(k,K)>0$ for arbitrary $K\in \bar I$ and $k<k_1$.
\item[$(vi)$] $k \mapsto f(k,K)$ is strictly concave for every $K\in \bar I$.
\end{itemize}
\end{assumption}

These assumptions can be easily seen to be satisfied by the Cobb-Douglas net output function $f(k,K)=Ak^\alpha K^\beta-\delta k$ for $0<\alpha,\beta<1$ and $\alpha+\beta<1$ with the only important observation that point {\it (v)} requires $\inf \bar I>0$.

\begin{assumption}[Assumptions on $U$] The utility function $U : \Rp \to \Rp$ is a concave, strictly increasing function such that $U'(0)=\infty$ and $U'(\infty)=0$. 
\end{assumption}

\begin{remark}
Explicit dependence in $z$ can be included both  in the net output function $f$ and/or the utiity function $U$, provided sufficient regularity conditions are imposed on such a dependence.
\end{remark}

\section{The individual optimisation problem \eqref{1}  }
\label{INDIV-OPT}

In this section we consider $K : \Rp  \to (0,\infty)$ (here the variable $z$ will be considered as fixed and will not play any role, hence  explicit dependence on $z$  will be omitted) as a given bounded function and consider the parametric optimisation problem
\begin{equation}\label{3}
\begin{aligned}
& \max_{c} \int_{0}^{\infty}e^{-rt} U(c(t)) dt, \,\,\, \mbox{subject to} \\
& k'(t) = f(k(t),K(t))-c(t) \\
& k(t) \ge 0, \,\,\, c(t) \ge 0,\,\,\,  \,\,\, a.e.
\end{aligned}
\end{equation}
We will also define the admissible set for the pairs $(k,c)$ 
\begin{eqnarray*}
\AK :=\{ (k,c) \,\,\, : \,\,\, k(t) \ge 0, \,\,\, c(t) \ge 0 \,\,\, k'(t)+ c(t) \le f(k(t),K(t)), \,\,\, t \in \Rp \},
\end{eqnarray*} 
postponing the specification of the  exact functional space setting for the pairs $(k,c)$ until the formulation of Lemma \ref{L1} where specific a priori bounds will be obtained, motivating our choice.

We point out that, thanks to the assumptions on $f$, whenever $K$ is bounded, then $t\mapsto f(k(t),K(t)$ is also bounded from above, by a constant that depends on $K$ and will be denoted by $M(K)=:C_1$.

We need to introduce the following functional spaces.  By $L^{p}(I)$, $p \in [1,\infty]$ we denote the standard Lebesgue spaces on $I \subset \Rp$ whereas $\Ml$ is the space of real valued Radon measures on $\Rp$, $\MII$ is the space of real valued finite Radon measures on $I \subset \Rp$, compact, the latter equipped by the norm
 $$\| \mu \|_{\MII}=|\mu |(I)=\sup \left\{  \int_{I} \varphi d\mu, \,\, ; \,\,\,  \varphi \in C(\R_+), \,\,\, \| \varphi \|_{\infty} \le 1 \right\}.$$
$BV(I)=\{  u : \Rp \to \R\,\, : \,\, u \in L^{1}(I), \,\, Du \in \MII\}$,  for every $I \subset \Rp$, which  is a Banach space when equipped with the norm $\| u\|_{BV(I)}:= \| u\|_{L^1(I)}+ \| Du \|_{\MII}$ and may be considered as an extension of $W^{1,1}(I)$, in the sense that the distributional derivative $Du$ is no longer an integrable function but rather a Radon measure, whereas $\BVl$ is the space of functions that belong to $BV(I)$ for every $I \subset \Rp$, compact.

The following lemma establishes some a priori bounds for the admissible set while at the same time indicates the proper functional space setting for the problem.

\begin{lemma}[A priori bounds for $\AK$] \label{L1}
Consider any pair $(k,c) \in \AK$.  Then,  $k \in L^\infty(\R_+)\cap \BVl$,  $c \in \Ml$, $k' \in \Ml$ and  for any compact subset  $I \subset \Rp$,  we have $\|c\|_{\MII} \le C$ and $\| k'\|_{\MII} \le C$, with the constant $C$ depending on the choice of $I$. In particular if $I=(0,t)$, we can choose the constant $C=C(t)=\bar C_0 + \bar C_1 t$ for suitable $\bar C_0,\bar C_1$.
 \end{lemma}

\begin{proof}  Let $(k,c) \in \AK$.  Consider first the constraint 
\begin{equation}\label{4}
k'+c \le f(k,K),
\end{equation}
which is considered in the weakest possible form, i.e. in the sense of measures.
Since $c$ is nonnegative and $f$ is negative for large $k$ (say for $k\geq \bar k$), then $k$ is bounded from above by $\max\{k(0),\bar k\}$. Hence, $k\in  L^\infty(\R_+)$.

We return to \eqref{4}, and note that since the right hand side is positive and bounded above by $C_1$,  we have that $k'(t) \le C_1$, for any $t$. This provides upper $L^\infty$ bounds for $k'$, but we also need lower bounds if we want to bound $k'$  in $\in \Ml$.
We break up $k'$ into its positive and negative part as $k'=\kdp - \kdm$. The positive part is bounded, by \eqref{4}, and the upper bound is $C_1$. We proceed to obtain a bound for $\kdm$. 
Clearly,
\begin{eqnarray*}
k(t)-k(0)=\int_{0}^{t}k' = \int_{0}^t \kdp - \int_{0}^{t}\kdm,
\end{eqnarray*}
which, using the fact that $k(t) \ge 0$ and the upper bound $ \kdp \le C_1$, leads to the estimate
\begin{eqnarray*}
-k(0) \le k(t)-k(0) = \int_{0}^{t} \kdp - \int_{0}^{t}\kdm \le C_1 \, t - \int_{0}^{t} \kdm,
\end{eqnarray*}
which upon rearrangement provides us with the estimate
\begin{eqnarray}\label{5}
\int_{0}^{t} \kdm \le k_{0} + C_1 \, t.
\end{eqnarray}
This allows us to estimate the total variation of $|k'|=\kdp + \kdm$ as
\begin{eqnarray}\label{6}
\int_{0}^{t} |k'| = \int_{0}\kdp + \int_{0} \kdm \le C_1 \, t + k_{0} + C_1 \, t =k_{0} + 2 C_1 \, t,
\end{eqnarray}
where we used  \eqref{5}. This estimate implies that $\| k'\|_{\MII} < C$ for any compact $I=[0,t]$, where $C$ depends on $t$.

Having obtained the desired bound for $k'$ we return once more to \eqref{4} which yields, $0 \le c \le C_1 - k'$, from which the desired bound on $c$ follows.
\end{proof}

We will now give a precise functional setting for Problem \eqref{3}. Let us define, for fixed $k_0\Rp$,
\begin{eqnarray*}
\AK :=\left\{ (k,c) \,\, : \,\, k\in  \BVl, \,c\in \Ml, \begin{array}{l}k \ge 0, \, k(0)=k_0,\, c \ge 0, \\ k'+ c \le f(k,K)\end{array} \right\},
\end{eqnarray*} 
where the last inequality is to be intended as an inequality between measures (the right-hand side being the measure on $\Rp$ with density $f(k(\cdot),K(\cdot))$). On the set of pairs $(k,c)$ (and in particular on $\AK$), we consider the following notion of (weak) convergence: we say $(k_n,c_n)\deb (k,c)$ if $c_n$ and $k'_n$ weakly-* converge to $c$ and $k'$, respectively, as measures on every compact interval $I\subset\Rp$. It is important to observe that, thanks to the fact that the initial value $k(0)$ is prescribed, the weak convergence as measures $k_n'\deb k'$ also implies $k_n\to k$ a.e., as
$$k_n(t)=k_0+\int_0^t k_n'(s)ds\to k_0+\int_0^t k'(s)ds=k(t)$$
for every $t$ which is not an atom for $k'$ (i.e. every continuity point $t$ for $k$). This condition is satisfied by all but a countable set of points $t$.

For $T\in \Rp\cup\{+\infty\}$, we also define $J_T\,:\, \Ml\to[0,+\infty]$ through
$$ J_T(c):= \int_{0}^{T}e^{-rt} U(c^{ac}(t)) dt, $$
where $c^{ac}$ is the absolutely continuous part of the measure $c$ (remember that every positive and locally finite measure $c$ on $\R$ can be decomposed as a sum $c^{ac}(t)dt+c^{sing}$, where the first part has density $c^{ac}$ w.r.t. the Lebesgue measure on $\R$, and the second part is singular: it could include atoms or other measures concentrated on sets with zero Lebesgue measure).

\begin{lemma}[Approximation and upper semicontinuity  of $J_{\infty}$]  \label{L2} The following hold:
\begin{itemize}
\item[(i)] For any $\epsilon >0$, there exists $T=T(\epsilon) < \infty$ such that $J_{\infty}(c) \le J_{T}(c)+\epsilon$ for any $(k,c) \in \AK$. .
\item[(ii)] $J_{\infty}$ is upper semicontinuous on $\AK$.
\end{itemize}
\end{lemma}

\begin{proof} (i)  We note that
\begin{eqnarray*}
J_{\infty}(c)=J_{T}(c)+\int_{T}^{\infty}e^{-r t} U(c^{ac}(t)) dt.
\end{eqnarray*}
For this part of the proof, we will write $c$ instead of $c^{ac}$ for simplicity.
By the additivity of the integral, the second  term is expressed as
\begin{eqnarray}\label{8}
\int_{T}^{\infty}e^{-r t} U(c(t)) dt =\sum_{k=0}^{\infty} \int_{T_{k}}^{T_{k+1}} e^{-r t} U(c(t)) dt
\end{eqnarray}
where $T_{k}=2^{k}T$, $k \in \N$. 
Since $U$ is concave,  it is bounded above by an affine function $U(c) \le C_3 + C_4 \,c $ with $C_3,C_4>0$,  so that for each $k \in \N$,
\begin{eqnarray*}
 \int_{ T_{k}}^{T_{k+1}} e^{-r t} U(c(t)) dt \le \int_{ T_{k}}^{T_{k+1}} e^{-r t}( C_3 + C_4 \, c(t)) dt \le e^{- r T_{k}} \int_{T_{k}}^{T_{k}} ( C_3 + C_4 \, |c|(t)) dt  \\
\le e^{-t T_{k}} (C_5 + C_6 T_{k+1}),
%
\end{eqnarray*}
where in the last inequality we used Lemma \ref{L1}.

Substituting this estimate into \eqref{8} we obtain
\begin{eqnarray*}
\int_{T}^{\infty}e^{-r t} U(c(t)) dt \le  C_5 \sum_{k=0}^{\infty} e^{-rT_{k}}  + C_6 \sum_{k=0}^{\infty} e^{-r T_{k}} T_{k+1} = C_5   \sum_{k=0}^{\infty} e^{-r T 2^{k}}  + 2 C_6 T \sum_{k=0}^{\infty} e^{-r T 2^{k}} 2^{k}\\
=C_5  e^{-r T} \sum_{k=0}^{\infty} e^{-r T (2^{k}-1)}  + C_6 e^{-r T} T\sum_{k=0}^{\infty} e^{-r T (2^{k}-1)} 2^{k}= : \varphi(T).
\end{eqnarray*}
It can be seen that $\varphi(T)< \infty$ and that $\lim_{T \to \infty} \varphi(T)=0$, from which the claim follows.

(ii)  Consider a sequence $\{c_n\} \in \AK$ such that $c_n \ws c$ in $\Ml$. By (i) for any $\epsilon >0$, there exists $T=T(\epsilon)$ such that 
\begin{eqnarray}\label{9}
J_{\infty}(c_n) \le J_{T}(c_n) + \epsilon, \,\,\, \forall \, n \in \N.
\end{eqnarray}
Furthermore,  using $U\ge 0$,
\begin{eqnarray}\label{10}
J_{T}(c) \le J_{\infty}(c), \,\,\, \forall \, c \in \Ml.
\end{eqnarray}
The functional $J_{T}$ is upper semicontinuous with respect to weak$-*$ convergence in $\Ml$. Indeed, since $U$ is concave, $-U$ is a convex function, and we can apply standard result on the semicontinuity of
$$c\mapsto \int h(c^{ac}(x))dx+\int h^{\infty}(c^{sing}),$$
where here $h=-U$ and $h^\infty=0$ (see, for instance
\cite{attouch2014variational}).  

Then, taking the limit supremum on both sides of \eqref{9}  yields
\begin{eqnarray*}
\lim\sup_{n} J_{\infty}(c_n) \le \lim\sup_{n} J_{T}(c_n) + \epsilon \le J_{T}(c) + \epsilon \le J_{\infty}(c)+ \epsilon
\end{eqnarray*}
where we used first the upper semicontinuity of $J_{T}$ and then \eqref{10}. Taking the limit as $\epsilon \to 0^{+}$ we obtain the upper semicontinuity of $J_{\infty}$.
\end{proof}

\begin{proposition}
Given any continuous and bounded function $K:\R_+\to\R_+$, the parametric maximization problem for the individuals \eqref{3} admits a maximiser $(k^{*},c^{*}) \in (\BVl \cap \Linl) \times \Ml$. 
\end{proposition}

\begin{proof}
Consider $(k_n,c_n)$ a maximizing sequence for \eqref{3}. Since $(k_n,c_n) \in \AK$ for any interval $I=[0,T]$, we conclude by Lemma \ref{L1} that $k_n$ is bounded in $\BVl$, and $c_n $ in $ \Ml$. Using the weak$-*$ compactness properties of  $\Ml$, this guarantees the existence of $k \in \BVl$ and $\zeta, c \in \Ml$ such that up to subsequences
\begin{eqnarray*}
k_n \to k & a.e. \\
k_n' \ws \zeta& \mbox{ in } \Ml, \\
c_n \ws c& \mbox{ in }  \Ml,
\end{eqnarray*}
where by standard arguments we see that $\zeta=k'$.

Our first task is to show admissibility of the limit $(k,c)$. It is straightforward to check that $c \ge 0$ and $k \ge 0$. It remains to check that $k'+c \le f(k,K)$. In order to see this, note first that since $(k_n,c_n) \in \AK$ for every $n \in \N$, we have that
\begin{eqnarray}\label{8bis}
k_n'+ c_n \le f(k_n,K), \,\,\, \forall \, n \in \N.
\end{eqnarray}
We need to pass to the limit as $n \to \infty$ (along the converging subsequence) in the above. The left hand side, is linear and 
$k_n'+c_n \ws k'+c$. Some care must be taken for the nonlinear right hand side. Since $f$ is continuous and bounded from above,  it is sufficient to use the pointwise convergence $k_n(t) \to k(t)$ a.e. in $t \in [0,T]$, guaranteed by the local BV bound on $k_n$.

To show that $(k,c)$ is a maximiser to Problem \eqref{3}, we use the upper semicontinuity result,  $\lim\sup_n J_{\infty}(c_n) \le J_{\infty}(c)$ contained in Lemma \ref{L2}. 
\end{proof}

The above result  guarantees the existence of a maximiser $(k,c)$ where both $k'$ and $c$ are measures. Before going on with the analysis, we will show that, actually, they have not singular parts, i.e. they are functions. In terms of the state function $k$, this means $k\in W^{1,1}_{loc}$. 
\begin{lemma}\label{cisac}
Any optimiser $(k,c)$ satifies $c \in L_{loc}^{1}(\R_{+})$.
\end{lemma}

\begin{proof} Let us recall that, in general, any $c$ such that $(k,c) \in \AK$ consists of an absolutely continuous and a singular part with respect to Lebesgue measure $c=c^{ac} + c^{sing}$. 

 For any Borel function $h$ we will use the notation $\avint_{a}^{b} h(s) ds =\frac{1}{b-a}\int_{a}^{b} h(s) ds$ (if $a \ne b$, otherwise this is equal to $h(a)$).

Suppose $(k,c)$ is optimal. Given an interval $I=[a,b]\subset\R_+$, we build $(\tilde k, \tilde c)$ a new competitor in the following way: $\tilde k$ is affine on $[a,b]$, connecting the values $k(a^-)$ to $k(b^+)$, and coincides with $k$ outside $[a,b]$; $\tilde c$ is absolutely continuous and constant on $[a,b]$. We have to choose the constant value for $\tilde c$ for the inequality $\tilde k'+\tilde c\leq f(\tilde k(t),t)$ to be satisfied. Notice that on $[a,b]$ we have $\tilde k'=k'([a,b])/(b-a)$. If we set $\hat c=c([a,b])/(b-a)$ for sure we have $\tilde k'+\hat c\leq \avint_a^b f(k(s),K(s))ds$. Since we can assume $f$ to be bounded ($k$ only takes values in a bounded set, and $\tilde k$ as well, hence), we have $\avint_a^b f(k(s),K(s))ds\leq f(\tilde k(t),t)+M$. Hence, it is sufficient to set $\tilde c=\hat c-M$ to have something admissible, provided $\hat c\geq M$.

Now, suppose $c$ is not absolutely continuous. Then for sure there exists $I=[a,b]$ such that $c^{sing}([a,b])> M(b-a)+c^{ac}([a,b])$ (otherwise $c^{sing}\leq M+c^{ac}$ and $c^{sing}$ would be absolutely continuous). The interval $[a,b]$ may be taken as small as we want which means that we can assume $(b-a)e^{-ra}\leq 2m_{a,b}$, where we set $m_{a,b}:=\int_a^b e^{-rs}ds$. 
We choose such an interval and we build $\tilde k$ and $\tilde c$ as above. We want to prove that $(\tilde{k},\tilde{c})$ is a better competitor.  Let us set $X=c^{ac}([b-a])/(b-a)$ and $Y=c^{s}([b-a])/(b-a)$. 

The optimality of $(k,c)$ implies
$$\int_a^b e^{-rt}U(c^{ac}(t))dt\geq \int_a^b e^{-rt}U(\tilde c(t))dt.$$
Yet, we have, using Jensen's inequality,
$$ \int_a^b e^{-rt}U(c^{ac}(t))\leq m_{a,b} U\left(\frac{1}{m_{a,b}}\int_a^b c^{ac}(t)e^{-rt}dt\right)\leq m_{a,b} U\left(\frac{e^{-ra}}{m_{a,b}}\int_a^b c^{ac}(t)dt\right)\leq  m_{a,b}U(2X)$$
and 
$$\int_a^b e^{-rt}U(\tilde c(t))dt=m_{a,b} U(X+Y-M).$$
We deduce 
$$U(X+Y-M)\leq U(2X),$$
but the interval $[a,b]$ was chosen so that $Y>X+M$, which gives a contradiction because of the strict monotonicity of $U$.\end{proof}

We can use the above result to prove some useful properties of our maximization problem.

\begin{corollary}\label{uniqamongall}
For any optimiser $(k,c)$, we have $k \in W^{1,1}_{loc}(\R_{+})$ and we have saturation of the constraint $k'+c=f(k,K)$. Moreover, we also have uniqueness of the optimal pair $(k,c)$. 
\end{corollary}
\begin{proof}
First, note that, if $k \notin W^{1,1}_{loc}(\R_{+})$, then this means that $k'$ has a negative singular part $((k')^{sing})_-$. Yet the inequality $k'+c\leq f(k,K)$ stays true if we replace $c$ with $c+((k')^{sing})_-$, which would be another optimiser. But this contradicts Lemma \ref{cisac}, since every optimiser $c$ should be absolutely continuous. 

Then, as soon as we know that both $k'$ and $c$ are absolutely continuous, it is clear that we must have $k'+c= f(k,K)$. Otherwise, we can replace $c$ with $f(k,K)-k'$ and get a better result (we needed to prove the absolute continuity of $k'$ as adding something to the singular part of $c$ does not improve the functional). 

In what concerns uniqueness, we first stress that the functional we maximise is not strictly concave, for two reasons: on the one hand we did not assume $U$ to be strictly concave, and on the other hand, anyway, there would be an issue as far as the singular part of $c$ is involved. Yet, now that we know that optimisers are absolutely continuous and saturate the differential inequality constraint, uniqueness follows from the strict concavity of $f$.

Indeed, suppose that that $(k_1,c_1)$ and $(k_2,c_2)$ are two minimisers. Then, setting $c_3=(c_1+c_2)/2$ and $k_3=(k_1+k_2)/2$, the pair $(k_3,c_3)$ is also admissible (since $f$ is concave in $k$) and optimal (since the cost is concave in $c$). Yet, strict concavity of $f$ implies that $(k_3,c_3)$ does not saturate the constraint, which is a contradiction with the first part of the statement, unless $k_1=k_2$. As a consequence, we obtain uniqueness of $k$ in the optimal pairs $(k,c)$. But this, together with the relation $k'+c= f(k,K)$, which is now saturated, also provides uniqueness of $c$.
\end{proof}

Now comes the most technical part of this section. We need to establish uniform lower bounds on the optimal $c$, which will be useful for approximation issues in the next section (roughly speaking: we need to compensate small perturbations in $f(k,K)$ by adding or subtracting small constants to $c$, but this could violate the positivity constraint $c\geq 0$ unless we have $c>0$, with uniform bounds). This requires to use the optimality conditions, in the form of a suitable version of the Pontryagin maximum principle, which should be adapted to our case: infinite horizon, and both state and control constraints. A recent version of the PMP, proved in \cite{de2009optimality}, exactly fits our needs, except the fact that it assumes all data to be smooth functions. This is in contradiction with our assumptions, which include $U'(0)=+\infty$; this difficulty will be handled  by approximation.

\begin{proposition}\label{lowerbound on c}
If $k_0>0$, then the optimal consumption $c$   is bounded from above and below on every bounded interval, i.e., for every $T>0$ there exist two constants $0<c_{-}(T)<c^+(T)<\infty$, depending only on $f, U, k_0$ and $T$, such that $c(t) \in [ c_{-}(T), c_{+}(T)]$ for every $t \in[0,T]$.
\end{proposition}

\begin{proof}  We first approximate $f$ by a sequence of functions $\{\fe\}$ such that $\fe \in C^{1}$ for every $h >0$ with $\fe \to f$ as $h \to 0$. Similarly for the utility function $U$ which is approximated by a sequence $\Ue \in C^1$.  We then consider Problem \eqref{3} with $f$ replaced by $\fe$ and $U$ replaced by $\Ue$. We will choose $\fe$ converging uniformly to $f$, with $\fe\geq f$, and take a specific form of $\Ue$: $\Ue(c):=U(c+h)$. We then invoke the characterization of the optimal path for this problem using the Pontryagin principle proposed in  \cite{de2009optimality}, for a fixed $h>0$.  

Fix a function $K$ and let $(\ke,\ce)$ be the corresponding optimal path for \eqref{3}, with $f$ replaced by $\fe$ and $U$ replaced by $\Ue$. Using Theorem 1 in \cite{de2009optimality}, there exists a nonnegative scalar $\lambda$,  a scalar $A \in \R$, an absolutely continuous function $p : \R_{+} \to \R$, a bounded variation function $q : \R_{+} \to \R$ and a nonnegative regular Borel measure $\mu$ defined on $\R_{+}$ such that
\begin{eqnarray}
&&p'(t)=-\theta(t) \, \fe_{k}(\ke(t),K(t)), \,\,\, p(0)=A,  \label{10-A}\\
&&q'(t)=r \lambda e^{-r t} \Ue(\ce(t)), \,\,\, \lim_{t \to \infty} q(t)=0,  \label{10-B}\\
&& \ce(t)\in \arg\max_{c_1 \ge 0} \{ \lambda e^{-r t} \Ue(c_1) + \theta(t) (\fe(\ke(t),K(t))-c_1) \} \,\, a.e. \,\, t\label{10-C} \\
&&\lambda e^{-r t} \Ue(\ce(t)) + \theta(t) (\fe(\ke(t),K(t))-\ce(t))=-q(t), \label{10-D} \\
&&(p,q,\lambda,\theta) \ne 0,  \label{10-E}
\end{eqnarray}
where $supp(\mu) \subset \{ t \,\, : \,\, \ke(t)=0\}$ and
\begin{eqnarray}\label{6-5-2016-B}
\theta'(t)=p'(t)-\mu(t).
\end{eqnarray}
Combining \eqref{10-A} with \eqref{6-5-2016-B} we obtain,
\begin{eqnarray}
\theta'(t)=-\theta(t) \, \fe_{k}(\ke(t),K(t)) -\mu.  \label{13-A}
\end{eqnarray}

We claim that $\lambda >0$. Assume per contra that $\lambda=0$. Then, \eqref{10-B} along with the boundary condition implies that $q(t)=0$ for all $t$. Also, \eqref{10-C} yields 
\begin{eqnarray*}
\max_{c_1 \ge 0} \{  \theta(t) \big(\fe(\ke(t),K(t))-c_1\big) \} =  \theta(t) (\fe(k(t),K(t))-\ce(t)),
\end{eqnarray*}
which in turn implies that $\max_{c_1 \ge 0} \{  - c_1 \theta(t)\}=-\theta(t) \ce(t)$, so that
\begin{eqnarray}\label{13-E}
\theta(t) \ce(t)=0, \,\,\, \theta(t) \ge 0, \,\,\, a.e \,\, t \in \R_{+}.
\end{eqnarray}
 
Suppose $\theta(0)=0$. In this case, from the equation $\theta'=p'-\mu=-\theta \fe_k -\mu\leq \delta\theta$, we get that $\theta(t)=0$ for every $t$. Then we can get $p'=0$, hence $p=A$. Looking at $\theta=p-\mu([0,t])$ for small $t$ (where $\mu$ vanishes, since $k\neq 0$), we deduce $A=0$, hence $p=0$. Now $\theta=p-\mu([0,t])$ implies $\mu=0$. Hence in this case we have $(\lambda,p,q,\mu)=0$ which is not possible.

Now suppose $\theta(0)>0$. Let $T_1=\inf\{t\,:\,\ke(t)=0\}$. On $[0,T_1[$, $\ke$ does not vanish, hence we have  $\theta'=-\theta \fe_k $. Since $\fe_k$ is bounded (we have approximated $f$ with $\fe$ on purpose), this equation implies that $\theta$ does not vanish on $[0,T_1[$. But in this case, we get $c=0$ on $[0,T_1[$. Hence, we have $(\ke)'(t)=\fe(\ke(t),K(t))$, and by the uniqueness of the solution of this equation (again, we have approximated, so that $\fe$ is Lipschitz now), it is not possible to have $T_1<+\infty$ and $\ke(T_1)=0$. Hence in this case $c(t)=0$ and $\ke(t)>0$ for every $t$. This last case can be excluded just by observing that zero-comsumption is never optimal as soon as $U$ is increasing and Lipschitz (just modify $c$ on any small interval and then start again using $c=0$).

Hence $\lambda>0$. Up to rescaling, we can suppose $\lambda=1$.

With $\lambda=1$, we now study the behaviour of $\theta$, and in particular we want to prove local upper bounds on $\theta$. Using again \eqref{13-A}, since  $\mu$ is nonnegative, and $\fe_{k}\ge -\delta$ for any $k, K$, we obtain the inequality
\begin{eqnarray}\label{11-C}
\theta'(t)\le \delta \theta(t),
\end{eqnarray}
so that by Gronwall  upper bounds on $\theta$ propagate into the future and lower bounds on $\theta$ propagate into the past. 

By \eqref{10-C}  in order to obtain upper bounds on $c(t)$ it is enough to obtain lower bounds on $\theta$: this is easy because, on the interval $[T,T+1]$, the bound on the mass of $c$ provides the existence of a point $t_0\in [T,T+1]$ where the value $c(t_0)$ can be bounded from above by a constant only depending on $f,U,k_0$ and $T$. This translates into a lower bound on $\theta(t_0)$ and, using \eqref{11-C}, on $\theta(t)$ for $t\in [0,T]$. This implies that $c$ is locally bounded on $\R^+$ and proves the upper bound part of the claim.

For the opposite bound, we just need to establish some upper bounds for $\theta$ on some initial time interval $[0,T_1]$. 

Under the assumption that $k_{0} >0$, and since $\ke \in W^{1,1}(\R_{+})$ is continuous, there exists a time $T_{1}=\inf\{ t \,\, : \,\, \ke(t)< k_2:=\min\{\frac{k_1}{2},\frac{k_0}{2}\}\}>0$, such that $\ke(t) \ge k_2>0$ for $t \in [0,T_1]$ (the value of $k_1$ used to define $k_2$ is the one appearing in \ref{f-ASSUMP}(v)).  Then $\mu(t)=0$ for any $t \in [0,T_1]$. Set 
$$C:=\sup_{k\ge k_2,t,h}  \fe_{k}(k,K(t)) <+\infty$$
(the fact that the sup is finite   relies on the fact that we look at $k$ far from $0$ and that $\fe$   will keep the same properties as $f$, and in particular Assumption   \ref{f-ASSUMP}(iii) ). Combining the above with  \eqref{13-A}, we get 
\begin{eqnarray*}
\theta'(t) \ge -C \theta(t), \,\,\, t \in [0,T_1].
\end{eqnarray*}
Hence, we conclude that   $\theta(t) e^{C t}$ is an increasing function of $t$ for $t \in [0,T_1]$.  Furthermore, by \eqref{11-C} we also have that $\theta(t) e^{-\delta t}$ is a decreasing function of $t$ for any $t \in \R_{+}$.   Consider now $t=T_1$. At this time instance $k$ decreases below the level $k_2$, so that  there exist points $t_n\to T_1^{-}$ such that $\ked(t_n)\leq 0$, and combining that with the binding dynamic constraint this yields that $f(\ke(t_n),K(t_n))-\ce(t_n) \leq 0$, from which it follows that
\begin{eqnarray*}
0<c_0\le \fe(k(t_n),K(t_n)) \le \ce(t_n),
\end{eqnarray*}
(with $k(t_n)\to k(T_1)=k_2$, $k_2\le k(t)\le k_1$) where we bounded $\fe$ from below by a uniform constant $c_0>0$ using Assumption \ref{f-ASSUMP}(v) ; since we use an approximation $\fe$ with $\fe\ge f$, the value of $c_0$ does not depend on $h$).

Whenever $\ce$ is bounded from below by a positive constant, this translates into an upper bound on $\theta$. Indeed, \eqref{10-C} provides 
$$\theta(t_n)=e^{-r t_n}\Ued(\ce(t_n)) \le e^{-r t_n} U'(\ce(t_n)),$$
as soon as $\ce(t)>0$. Using the behaviour of $ U'$, and the fact that $t_n\to T_1$, this gives an upper bound on $e^{rT_1}\theta(T_1)$. 

Coupling this last bound with the fact that  $\theta(t) e^{C t}$ is increasing for $t \in [0,T_1]$, this gives a bound on $\theta$:
$$\theta(t)\leq \theta(T_1)e^{C (T_1-t)}.$$
This can be used to bound $\theta$ on an initial interval, but this bound depends on $T_1$.
%

We now need to consider the question of how large $T_1$ actually is. Pick any $\epsilon >0$,  and applying Lemma \ref{L2}  (which obviously holds for the
 problem \eqref{3} with $f$ replaced by $\fe$ and $U$ replaced by $\Ue$ and using the notation $J^{(h)}_{T}$ and $J^{(h)}_{\infty}$ for the finite horizon and infinite horizon value functions respectively)
 there exists $T_0=T(\epsilon) < \infty$ such that $J_{\infty}(c) \le J_{T_0}(c)+\epsilon$. Either $T_1 <T_0$ or $T_1 \ge T_0$.
 
In the first case, $T_1$ is bounded. This is the easy case, as it provides a uniform upper bound on $\theta$. Then, using  
$$\theta(t)\geq e^{-r t}\Ued(\ce(t)),$$
(which is also valid when $\ce=0$), we obtain lower bounds on $\ce$.

We have to look at the second case, suppose $T_1\geq T_0$, and suppose that $\inf_{t\in [0,T_0]} \ce(t)$ is very small, say $\inf_{t\in [0,T_0]} \ce(t)<\epsilon_0$. We will show that this cannot hold if $\epsilon_0$ is small enough. 

In the current case, on the interval $[0,T_0]$, the function $\theta$ is a function which does not vary too much. More precisely, on this interval both $\theta(t) e^{C t}$ in increasing and $\theta(t) e^{-\delta t}$ is decreasing. Moreover, $\ce(t)$ can be deduced from $\theta(t)$. Using \eqref{10-C} we have 
$$\ce(t)=(\Ued)^{-1}(\min\{\theta(t)e^{r t},\Ued(0)\}),$$
which means, thanks to the precise form of $\Ue$
$$\ce(t)=(U')^{-1}(\min\{\theta(t)e^{r t},U'(h)\})-h.$$

Now, suppose $\inf_{t\in [0,T_0]} \ce(t)<\epsilon_0$. Then there exists a $t_0\in [0,T_0]$ such that $\theta(t_0)e^{ -rt_0}>U'(\epsilon_0+h)$. This means that for every $t \in [0,T_0]$ we have $\theta(t)e^{-rt}>U'(\epsilon_0+h)e^{-CT_0}$.  But this implies 
$$\ce(t)<(U')^{-1}(\min\{U'(\epsilon_0+h)e^{-CT_0},U'(h)\})-h$$
for every $t \in [0,T_0]$, and hence 
$$\Ue(\ce(t))<U((U')^{-1}(\min\{U'(\epsilon_0+h)e^{-CT_0},U'(h)\})).$$

This provides 
$$J^{(h)}_{T_0}(\ce)<U((U')^{-1}(\min\{U'(\epsilon_0+h)e^{-CT_0},U'(h)\}))\int_0^T e^{-rt}dt.$$

Our goal is to prove that there exists $\epsilon_0>0$ such that this is not possible, at least for small $h$. If this is not the case, then we can actually pass to the limit in the above inequality with $\epsilon_0,h\to 0$.
Note that 
$$U((U')^{-1}(\min\{U'(\epsilon_0+h)e^{-CT_0},U'(h)\}))=\max\{U(h),U((U')^{-1}(U'(\epsilon_0+h)e^{-CT_0}))\}\to U(0)$$ 
as $\epsilon_0, h\to 0$ (we need to use $U'(0)=+\infty)$).

  We claim the convergence of $J^{(h)}_\infty(\ce)=\max J^{(h)}_\infty$ to $\max J_\infty$, (for the proof see Lemma \ref{ch}). But this gives 
$$\max J_\infty=\lim_{h\to 0} J^{(h)}_\infty(\ce)\leq \liminf_{h\to 0}J^{(h)}_{T_0}(\ce)+\epsilon\leq \epsilon + U(0)\int_0^T e^{-rt}dt\leq \epsilon + U(0)\int_0^\infty e^{-rt}dt,$$
which is a contradiction, for small $\ve>0$, because (just by using the non-optimality of the zero-consumption scenario $c=0$), we always have $\max J_\infty>U(0)\int_0^\infty e^{-rt}dt$. 
Hence, on the interval $[0,T_0]$ we can conclude that we have a uniform lower bound on $\ce(t)$, which will be different from $0$, and this will turn into an upper bound for $\theta(t)$. We already observed that this uppper bounds propagates to the future, and implies a lower bound on $\ce$ on $[0,T]$. 

We have therefore concluded that for every $T$ there exists a value $c_{-}(T)$ such that for any $h>0$ the optimal path $\ce$ of the regularised problem has the property $\ce(t)  \ge c_{-} (T)>0$ for all $t \in[0,T_0]$. We now pass to the limit as $h \to 0$ and conclude that the desired property holds for the original problem as well. This can be done using the following Lemma \ref{ch}.\end{proof}

\begin{lemma}\label{ch}
With our assumptions $\fe\to f$, $\fe\ge f$, $\Ue(c)=U(c+h)$, we have
$$\max J_\infty=\lim_{h\to 0} \max J^{(h)}_\infty$$
and $(\ke,\ce)\deb (k,c)$ (in the sense of the convergence that we defined for pairs $(k,c)$, where $(\ke,\ce)$ and $ (k,c)$ are the optimisers of $J_\infty$ and $J^{(h)}_\infty$, respectively. \end{lemma}
\begin{proof}
Let us consider the minimisers $(\ke,\ce)$: they satisfy the same uniform bounds in the space of measures that we saw in Lemma \ref{L1}, and hence we can assume, up to subsequences, that they satisfy  $(\ke,\ce)\deb (\tilde k,\tilde c)$, but we do not know yet that the limit is the maximiser for the limit problem. However, we can pass to the limit the inequality $\ked + \ce \leq \fe(\ke,K)$ and get the admissibility of $(\tilde k,\tilde c)$. We can also apply the semicontinuity result of Lemma \ref{L2} to $\ce+h\deb \tilde c$, thus getting 
$$J_\infty(\tilde c)\geq \limsup_{h\to 0} J^{(h)}_\infty(\ce).$$
This proves that $\max J_\infty \geq J_\infty(\tilde c)\geq \limsup_{h\to 0} \max J^{(h)}_\infty$, and $(\tilde k,\tilde c)$ is optimal if we can prove 
$$ \max J_\infty\leq \liminf_{h\to 0} \max J^{(h)}_\infty.$$
Let $(k,c)$ be the (unique) maximiser for $J_\infty$: we can choose to use it also in the problem with $J^{(h)}_\infty$, thanks to $f_h\geq f$. Thus we get
$$ \liminf_{h\to 0} \max J^{(h)}_\infty\geq  \liminf_{h\to 0} J^{(h)}_\infty(c)= \liminf_{h\to 0}\int U(c+h)e^{-rt}dt\geq \int U(c)e^{-rt}dt=J_\infty(c)= \max J_\infty$$
(in the second inequality we used Fatou's lemma and positivity of $U$). This proves that $\tilde c$ is optimal, hence $\tilde c=c$ by uniqueness and the whole sequence converges.
\end{proof}

Some corollaries of the previous result are the following
\begin{corollary}\label{kLip}
The optimal $k$ is locally Lipschitz on each interval $[0,T]$, and its Lipschitz constant only depends on $U,f,k_0$ and $T$.
\end{corollary}

\begin{proof} The saturation of the constraint gives $k'=f(k,K)-c$. From the boundedness of $k$ and $K$ and the local upper bound on $c$ obtained from
Proposition \ref{lowerbound on c}, we obtain a local $L^\infty$ bound on $k'$, which proves the claim.
\end{proof}

\begin{corollary}\label{k non zero}
The optimal $k$ satisfies $k(t)>0$ for every $t\in\Rp$. Moreover, the optimal $(k,c)$ satisfies the differential equations
$$\begin{cases} k'=f(k,K)-c\\
			U''(c)c'=(r-f_k(k,K))U'(c)\end{cases}$$	
\end{corollary}

\begin{proof} Thie condition $k>0$ is a consequence of the lower bound on $c$. Suppose that there exists an instant $t_0$ such that $k(t_0)=0$. Take the lower bound $c_{-}(2t_0)$ for $c$ on the interval $[0,2t_0]$. Since $f(0,K)=0$ and $f$ and $k$ are continuous, on a small neighbourhood $(t_0-\epsilon,t_0+\epsilon)$ of $t_0$ we have $f(k(t),K(t))\leq c_{-}(2t_0)/2$. Then, on this interval we have $k'(t)\leq -c_{-}(2t_0)/2<0$. But this is impossible since $k\geq 0$ and $k(t_0)=0$.

Once we know that neither $c$ nor $k$ vanish, then standard Pontryagin principle (or the conditions provided by Theorem 1 in \cite{de2009optimality} in the case where there is no measure $\mu$); or standard Euler-Legrange equations in calculus of variations, considering that now $f$ and $U$ are smooth) provide the desired optimality conditions in the form of the desired system of ODEs.
\end{proof}

\section{Existence of rational expectations equilibria}
\label{RAT-EXP}

We now turn our attention to the rational expectations equilibria. To this end, we recall that we defined the operator $\T$  as follows: Consider any externalities configuration $K$, solve the parametric individual optimisation problem \eqref{3} for any $z \in D$ to obtain the family of  maximisers $\{k^{*}(\cdot, z ; K)\}_{z \in D}$ and then act the operator $\S$ on the maximiser to obtain $\S k^{*}$, i.e.,
\begin{eqnarray*}
K \mapsto k^{*} \mapsto \S k^{*} =: K^{\#}, \,\,\,\,\, \T\, K: =  K^{\#}.
\end{eqnarray*}
 
A fixed point for the operator $\T$ is identified as a rational expectations equilibrium.  

Our proof of existence is based upon the Schauder fixed point theorem and requires the continuity of the operator $\T$.  Note that by defining in the above scheme first the operator $\M$ by $\M K := k^{*}$, we may decompose the operator $\T$ as $\T=\S \,  \M$. Since $\S$ is a compact operator, we focus our attention on the properties of the operator $\M$. As the operator $\M$ is defined by mapping the parameter $K$ of the parametric optimisation problem \eqref{3} to the optimiser,  in order to obtain the continuity of $\M$ we need to ensure the continuity of maximisers with respect to the parameter. This is essentially a $\Gamma$-convergence type result (see e.g. \cite{attouch2014variational} or \cite{dal2012introduction}), but we will avoid using this theory, as we will directly consider optimisers. The argument is very similar to what developed in Lemma \ref{ch}.

Consider a sequence $\{K_n\} \subset L^{\infty}(\Rp)$ (uniformly bounded in $n \in \N$) such that $K_{n} \to K$, locally uniformly, and the corresponding sequence of  constraint sets
\begin{eqnarray*}
\AKn =\{ (k,c) \,\,\, : \,\,\, k'+c \le f(k,K_n)\},
\end{eqnarray*}
and  let  $(k_n^{*},c_n^{*})$ be the maximisers of  the sequence of optimisation problems  
\begin{eqnarray*}
\max_{(k,c) \in  \AKn}J_{\infty}(c).
\end{eqnarray*}
We need to show that we have $(k_n^{*},c_n^{*})\to (k^{*},c^{*})$, where $(k^{*},c^{*})$ is the maximiser of the optimisation problem 
\begin{eqnarray*}
\max_{(k,c) \in \AK} J_{\infty}(c).
\end{eqnarray*}
Note that the above maximiser is unique if the conditions of Corollary \ref{uniqamongall} are satisfied.

\begin{proposition}\label{convKnK}
Suppose that 
 $\{K_n\} \subset L^{\infty}(\Rp)$ is a uniformly bounded sequence such that $K_{n} \to K$ locally uniformly, and let  $(k_n^{*},c_n^{*})$ be the corresponding maximisers. Then we have $(k_n^{*},c_n^{*})\to (k^{*},c^{*})$, where $(k^{*},c^{*})$ is the maximiser corresponding to $K$. 
\end{proposition}
\begin{proof}
We follow a similar scheme as in Lemma \ref{ch}.

The functions  $(k_n^{*},c_n^{*})$ satisfy the same uniform bounds in the space of measures that we saw in Lemma \ref{L1}, and hence we can assume, up to subsequences, that they satisfy  $(k_n^*,c_n^*)\deb (\tilde k,\tilde c)$, but we do not know yet that the limit is the maximiser for the limit problem. Yet, we can pass to the limit the inequality $(k_n^*)'+c_n^*\le f(k_n^*,K_n)$ and get the admissibility of $(\tilde k,\tilde c)$. We can also apply the semicontinuity result of Lemma \ref{L2} to, thus getting 
$$J_\infty(\tilde c)\geq \limsup_{n} J_\infty(c_n^*).$$
This proves that 
$$\max \{ J_\infty(c)\,:\, (k,c)\in\mathcal A(K)\} \geq J_\infty(\tilde c)\geq \limsup_{n} J_\infty(c_n^*)=\limsup_{n} \max \{ J_\infty(c)\,:\, (k,c)\in\mathcal A(K_n)\}. $$
Then, $(\tilde k,\tilde c)$ is optimal if we can prove $\max \{ J_\infty(c)\,:\, (k,c)\in\mathcal A(K)\} \leq \limsup_{n} \max \{ J_\infty(c)\,:\, (k,c)\in\mathcal A(K_n)\}.$

In order to do so, let us take the optimal pair $(k^*,c^*)$, fix $\epsilon>0$ and find $T$ such that  $J_\infty(c^*)\leq J_T(c^*)+\epsilon$.
Then, let us define a sequence $(k_n,c_n)\in \mathcal A(K_n)$ as follows. Set $\epsilon_n:=||f(k^*,K)-f(k^*,K_n)||_{L^\infty([0,T])} \to 0$ and use
$$k_n(t)=k^*(t),\;c_n(t):=c^*(t)-\epsilon_n\quad \mbox{for }t\in [0,T];$$
for $t>T$, just use any admissible pair $(k,c)$ satisfying $k(T)=k^*(T)$ and all the constraints.
It is important to use Proposition \ref{lowerbound on c} to guarantee that, at least for large $n$, the consumption $c_n$ defined above is admissible (i.e. nonnegative). 

Then, by monotone convergence we have
$$J_T(c^*)=\lim_n J_T(c_n)\leq \limsup_n J_\infty(c_n)\leq  \limsup_{n} \max \{ J_\infty(c)\,:\, (k,c)\in\mathcal A(K_n)\},$$
(where we use positivity of $U$, to pass from $J_T$ to $J_\infty$).

We deduce 
$$J_\infty(c^*)\leq  \limsup_{n} \max \{ J_\infty(c)\,:\, (k,c)\in\mathcal A(K_n)\}+\epsilon$$
and the result is proven since $\epsilon$ is arbitrary. 

The optimality of $(\tilde k,\tilde c)$ together with the uniqueness of the maximer implies $(\tilde k,\tilde c)=(k^*,c^*)$ and the convergence of the full sequence without the need to extract a subsequence is standard.
\end{proof}

We now define a space $Y$ of continuous and locally bounded functions of the variables $(t,z)\in\R_+\times D$ in the following way: set
$$||y||_Y:=\sum_{m\geq 0} 2^{-m}||y||_{L^\infty([0,m]\times D)},$$
and then  use 
$$Y=\{y\in L^\infty_{loc}(\R_+\times D)\,:\, ||y||_Y<+\infty\},$$
endowed with the norm $||\cdot||_Y$. This is a Banach space, and we want to use Schauder's fixed point theorem in it. Whenever a sequence $y_n$ is such that $\| y_n\|_{\infty}$ is bounded, then $y_n\to y$ in $Y$ if and only if $y_n \to y$ uniformly on  sets of the form  $[0,m] \times D$, i.e. if and only if $y_n \to y$ uniformly in $D$ and locally uniformly in time $t\in\Rp$.

The space $Y$ is the space where externalities $K$ reside. Externalities are typically supposed to be uniformly bounded (and so they actually belong to $Y$, as $L^\infty(\R_+\times D)\subset Y$).

The operator $\T$ will be defined on a subset of $Y$, precisely on $L^\infty(\R_+\times D)\subset Y$. As we said, it is  factorised as $\T=\S \M$, where
 $$(\M K)(\cdot,z):=\mbox{ the unique optimal curve $k^*$ associated with the externality $K(\cdot,z)$}$$
 and $\S : (\BVl \cap \Lil) \to Y$ is the integral operator defined in \eqref{defi SS}, which is essentially not affecting time but only space.   Because of Proposition \ref{convKnK}, the operator $\M$ is continuous, and hence is $\T$. 
 We want to use Schauder's fixed point theorem on the operator $\T$, thus obtaining existence of a rational expectations equilibrium.
 \begin{theorem}
 Under the standing assumptions of this paper, and the assumption $\inf_z k_0(z)>0$, the operator $\T$ admits a fixed point in $Y$, and hence there exists a rational expectation equilibrium for the spatial economy. 
 \end{theorem}
 \begin{proof}
 Consider $Y_0\subset Y$ the set of continuous functions on $R_+\times D$ bounded by a constant $M$ such that 
 $$\mathrm{Lip}(\S)||\max\{\sup_z |k_0(z)|,\bar k(M)\}+||\S(0)||_{L^\infty}<M$$ 
 (here $\S(0)$ is the function obtained by applying $\S$ to the zero function, and we use the Lipschitz behaviour of $\S$ so that for every bounded function $g$ we have $||\S (g)||_{L^\infty}\le \mathrm{Lip}(\S) ||g||_{L^\infty})+||\S(0)||_{L^\infty}$). With this choice (which is always possible choosing $M$ large enough, since we supposed $\bar k(M)=o(M)$), the operator $\T$ maps $Y_0$ into itself, since for every continuous function $K(\cdot)$ bounded by $M$ the corresponding optimiser $k^*(\cdot)$ takes values in $[0,\max\{k_0,\bar k(M)\}]$.
 
 As $Y_0$ is a convex subset of $Y$ on which $\T$ is defined and continous, in order to apply Schauder's fixed point theorem we just need to check that $\T$ maps $Y_0$ into a compact subset of $Y_0$. For our choice of convergence, we just need to guarantee equicontinuity of the functions of the form $\T( K)$ on each set $[0,T]\times D$. The continuity in time of each maximiser $k_*$ is guaranteed by Corollary \ref{kLip}, which provides Lipschitz bounds in time of $k^*$, combined with the Lipschitz behaviour of $\mathcal S$. The continuity in space is guaranteed by the assumptions on $\S$.
 
 As a consequence, $\T$ admits a fixed point.
 \end{proof}
 
 \begin{theorem} The rational expectation equilibrium $(k,c,K)$ is characterised by the following non-local Euler-Lagrange equation
 $$\begin{cases} k'=f(k,\S(k))-c,\\
			U''(c)c'=(r-f_k(k,\S(k)))U'(c),\\
			K=\S(k).\end{cases}$$	
\end{theorem}
\begin{proof}
It is easy to see, using Corollary \ref{k non zero}, that if  $(k,c,K)$ is a rational expectation equilibrium, then it satisfies this system. On the other hand, if $(k,c,K)$ satisfies the system, then we can solve the individual optimisation problems for given $K=\S(k)$ and find a solution $k^*$ which is characterised (because of the concavity of the maximization problem) by the equations in Corollary \ref{k non zero}. This means that $k^*=k$, and hence $k=\M (K)$ and $K=\S(\M(K))$ is a fixed point.
\end{proof}

 \begin{remark} The arguments can be generalised in the case where the integral operator that models the effect of externalities, $\S$
acts on the temporal variable as well on the spatial variable, for instance
$(\S k)(t,z)=\int_{D} \int_{0}^{\infty} w(z,y,t,s) k(y,s) dy ds$,
where now $\S$ can have some sort of regularization effect on the temporal variable, or include memory effects or delays. Thisis particularly meaningful in order to model the effect of knowledge spillover in the generation of externalities, as we mentioned in the introduction. From the technical mathematical point of view, this requires minimal adaptations.\end{remark}

\begin{remark} It is easy to see that, in case of lack of uniqueness for the maximisers of $J_\infty$ in $\mathcal{A}(K)$, then the proof of Proposition \ref{convKnK} easily provides that, up to extracting a subsequence, we have convergence of $(k_n^{*},c_n^{*})$ to one maximiser. This makes the operator $\M$ a multivalued map with closed graph. Moreover, if $f(\cdot,K)$ is concave, the set of maximisers is a convex set, which allows to obtain the existence of a fixed point by an easy application fo the Kakutani Theorem.
\end{remark}

%

\bibliographystyle{plain}
\bibliography{fs-any-references}
\end{document}